\documentclass[a4paper, 12pt, reqno]{amsart} 

\usepackage{amsmath,amsfonts,amssymb,amsthm}
\usepackage{enumerate,enumitem}
\usepackage{xcolor}

\usepackage{geometry}
\usepackage[utf8]{inputenc}
\usepackage[alphabetic,nobysame]{amsrefs}
\newcommand{\citeref}[2]{\cite{#1}*{#2}}

\usepackage[hidelinks]{hyperref} %remove option hidelinks to make the internal links surrounded by a rectangle

\newif\ifdraft
\draftfalse %uncomment this for the final version (arxiv or journal)

\newif\ifarxiv
\arxivfalse %uncomment this in the journal version

\ifdraft % only for debugging
  \usepackage[notcite,notref]{showkeys}
  \else
\fi % only for debugging
\newcommand\E{{\mathbb E}}
\newcommand\Ec[2]{{\E[#1\,|\,#2]}}
\renewcommand\Pr{{\mathbb P}}

\newcommand\norm[1]{\left\lVert #1 \right\rVert}

\newcommand{\R}{\mathbb{R}}

\newcommand{\e}{{\mathrm{e}}} 
\newcommand{\vv}{{\mathbf v}}
\newcommand{\xx}{{\mathbf x}}

\newtheorem{theorem}{Theorem}

\newtheorem{lemma}[theorem]{Lemma}

\author{Tomas Ju\v{s}kevi\v{c}ius}
\address{Institute of Computer Science, Vilnius University, Didlaukio g. 47, 08303 Vilnius, Lithuania}
\email{tomas.juskevicius@gmail.com}
\author{Matas \v{S}ileikis}
\title{Concentration inequalities for vector-valued martingales}
\address{Institute of Computer Science of the Czech Academy of Sciences, Pod Vod\'{a}renskou v\v{e}\v{z}\'{\i} 2, 182 00 Prague, Czechia; Institute of Computer Science, Vilnius University, Didlaukio g. 47, 08303 Vilnius, Lithuania. E-mail: {\tt matas@cs.cas.cz}.} 

\date{\today}

\begin{document}

\begin{abstract}
In this paper, we establish several deviation inequalities for the norms of vector-valued martingales, achieving bounds of the same quality as those available in the classical one-dimensional setting. Our results extend and refine previous work by Pinelis (1994) and Hayes (2005) in this area.
\end{abstract}

\maketitle  

\section{Introduction}

Let $(M_n)_{n\ge 0}$ be a martingale taking values in a Euclidean
space, with $M_0=0$, and suppose that its differences
$X_k := M_k-M_{k-1}$ satisfy $\|X_k\|\leq a_k$ almost surely for
all $k$. As an extension of the classical
Azuma--Hoeffding inequality, Pinelis
\citeref{Pinelis1994martingales}{Theorem 3.5} established the bound
\begin{equation}
    \Pr(\|M_n\|\geq x)
    \leq 2\exp\left(-\frac{x^2}{2\sum_{i=1}^n a_i^2}\right).
    \label{eq:Pinelis}
\end{equation}
The aim of this paper is to refine \eqref{eq:Pinelis} in two
directions.

First, we establish a comparison with the tail of a Gaussian random
variable, yielding a more precise description of the deviation
probabilities.

\begin{theorem}\label{thm:subgaussian_tail}
For every $x\geq 0$,
\[
    \Pr(\|M_n\|\geq x)
    \leq \frac{2\e^3}{9}\Pr(|Z_n|\geq x),
\]
where $Z_n\sim N(0,\sigma_n^2)$ and
$\sigma_n^2=\sum_{i=1}^n a_i^2$.
\end{theorem}

Second, we investigate the large-deviation regime, in which $x$ is
of order $n$ when the bounds $a_i$ remain of constant order. In this
regime, a Gaussian comparison does not capture the correct
large-deviation rate. In particular, the Gaussian tail has a
quadratic exponent, whereas the relevant rate in the large-deviation
regime is described by a different function. Our second result
 provides a bound that captures this rate. The exponent is asymptotically sharp, which follows by considering a rescaled sum of i.i.d. Rademacher random variables.

\begin{theorem}\label{thm:largedeviation_tail}
Suppose that $a=\sqrt{\frac{1}{n}\sum_{i=1}^n a_i^2}>0$. For $0\leq x\leq na$, we have
\[
    \Pr(\|M_n\|\geq x)
    \leq
    2\exp\left(
        -{n}h\left(\frac{x}{na}\right)
    \right),
\]
where $h$ is a continuous function defined by %did not specify the domain so that pedantic person does not argue that it's undefined for $u = 0, 1$ and understands it is extended to ends by continuity
\[
    h(u)=\tfrac 12[(1+u)\ln(1+u)+(1-u)\ln(1-u)].
\]

\end{theorem}
For $x>na$, the probability on the left-hand side is zero, since $\|M_n\|\leq\sum_{i=1}^n a_i\leq na$ almost surely.

The paper is organized as follows. Section~2 discusses the main
obstacles to extending one-dimensional deviation inequalities to
higher-dimensional martingales. Section~3 contains the technical
lemmas needed in the proofs of our main results. Finally,
Section~4 is devoted to the proofs of
Theorems~\ref{thm:subgaussian_tail} and
\ref{thm:largedeviation_tail}.

\section{Difficulties in higher dimensions}

The deviation inequalities discussed in the introduction rely on
the possibility of comparing suitable functions of the norm of a
vector-valued martingale with the corresponding functions of a
one-dimensional weighted sum of Rademacher random variables. More precisely, for an appropriate
parametrized family of functions $\{f_h\}$, one seeks to establish
the inequality
\begin{equation}
    \mathbb{E} f_h(\|M_n\|)
    \leq
    \mathbb{E} f_h\left(
        \left|\sum_{i=1}^n a_i\xi_i\right|
    \right),
    \label{1d}
\end{equation}
where $\xi_1,\dots,\xi_n$ are independent Rademacher random
variables, that is,
$\Pr(\xi_i=1)=\Pr(\xi_i=-1)=1/2$.
Such a comparison shows that the relevant extremal bounds for
vector-valued martingales can be obtained from a suitable
one-dimensional example.

The classical exponential inequalities are derived by considering
the family $f_h(t)=e^{ht}$, with $h\geq0$. However, as observed by
Hayes \cite{Hayes2005}, the quantity
\[
    \mathbb{E}e^{h\|M_n\|}
\]
is not maximized by one-dimensional martingales. This motivates the search
for other classes of functions for which the comparison
\eqref{1d} remains valid.

A useful choice is provided by functions whose Taylor expansions
contain only even powers. Indeed, the moments
$\mathbb{E}\|M_n\|^{2k}$, for $k\in\mathbb{N}$, are uniformly
bounded by the corresponding moments of the one-dimensional
Rademacher sum. This moment comparison provides a route to
establishing deviation inequalities for $\|M_n\|$ that retain the
quality of their one-dimensional counterparts. A similar route was taken in \cite{Pinelis1994martingales} where he used the functions $$f_{h}(t)=\cosh(ht)=\sum_{k=0}^{\infty}\frac{(ht)^{2k}}{(2k)!}.$$ to establish \eqref{eq:Pinelis}. 

\bigskip

We should also like to mention that we make use of a result by Kallenberg and Sztencel \cite{kallenberg} to reduce the problem to two dimensions. We note that for our proofs this reduction is not essential, but it somewhat simplifies the notation and is worth mentioning in the context.

\section{Inequalities for convex functions of the norm in \texorpdfstring{$\mathbb{R}^2$}{R²}}

The goal of this section is to establish some elementary lemmas that are needed to establish the main technical statement of the paper, that is, Theorem \ref{thm:moments}. We shall deal only with martingales in $\mathbb{R}^2$ as in the end the reduction to this case will follow from a result by Kallenberg and Sztencel \cite{kallenberg}. The latter reduction is discussed in Section 4 before the proofs of our concentration inequalities.

The following statement is the main comparison inequality from which we will derive Theorems \ref{thm:subgaussian_tail} and \ref{thm:largedeviation_tail}. 
Let $\xi_1, \dots, \xi_n$ be independent Rademacher random variables. By $B_{r} \subset \R^2$ we denote the closed disk of radius $r$ centered at the origin.
\begin{theorem}
  \label{thm:moments}
  Let $(M_n)_{n \ge 0}$ be a $\R^2$-valued martingale with $M_0 = 0$ such that every difference $X_i = M_i - M_{i-1}, i = 1, \dots, n$, satisfies inequality $\norm{X_i} \le a_i$ a.s. Let $f : \R \to \R$ be an even, twice continuously differentiable convex function whose second derivative $f''$ is convex. We have
    \begin{equation}
      \E f(\norm{M_n}) \le \E f(a_1\xi_1 + \dots + a_n\xi_n).
    \end{equation}
\end{theorem}

We defer the proof of Theorem \ref{thm:moments} to the end of this section as we need to develop some useful lemmas before providing this proof. The core Lemma~\ref{lem:single} asserts that among random vectors shifted by a fixed vector $\vv$, the $f$-moment is maximized by a two-point distribution collinear with $\vv$. 

%\begin{lemma}
%  \label{lem:conv_even}
%    If a random variable $X$ has a symmetric distribution, that is $X  \overset{d}{=} -X$, and $\psi : \R \to \R$ is an even convex function then
%    \[
%      f(t) := \E \psi(t + X)
%    \]
%    is convex and even.
%\end{lemma}
%\begin{proof}
%Since \(X\overset{d}{=} -X\) and $\psi$ is even, for any $t \in \R$
%\[
%f(-t) = \E \psi (-t + X) = \E \psi(t - X) =\E \psi(t+X)=f(t).
%\]
%Therefore  \(f\) is even.
%
%For convexity, let \(t_1,t_2\in\mathbb{R}\) and let \(\lambda\in[0,1]\). For every  value of \(X\) the map \(t\mapsto \psi(t+X)\) is convex. Hence 
%\[
%\psi\big(\lambda t_1+(1-\lambda)t_2 + X\big)
%\le \lambda\,\psi\big(t_1+X\big) + (1-\lambda)\,\psi\big(t_2+X\big).
%\]
%Taking expectations yields
%\[
%\mathbb{E}\psi(\lambda t_1+(1-\lambda)t_2 + X)
%\le \lambda\,\mathbb{E}\psi(t_1+X) + (1-\lambda)\,\mathbb{E}\psi(t_2+X).
%\]
%That is,
%\[
%f\big(\lambda t_1+(1-\lambda)t_2\big)\le \lambda f(t_1) + (1-\lambda) f(t_2),
%\]
%so \(f\) is convex.
%
%\end{proof}

\begin{lemma}
  \label{lem:increasing}
  If $f : \R \to \R$ is an even, twice continuously differentiable function whose second derivative $f''$ is convex, then function $h(t) = f'(t)/t$ is increasing on $(0,\infty)$. %the proof gives that if $(-a,a)$ is the maximal interval on which $f''$ is constant, then $f(t)/t$ is constant on $[-a,a]$ and strictly increasing on $[a, \infty)$
\end{lemma}
\begin{proof}
  Since $f$ is even, $f'$ is odd, so $f'(0) = 0$. Moreover, $f''$ is even and, by assumption, convex and therefore increasing on $[0, \infty)$.
    Consequently for $t \ge 0$
    \[
      f'(t) = \int_0^t f''(u) du \le t f''(t).
    \]
    Therefore, for $t > 0$ we have
    \begin{equation}
      h'(t) = \frac{tf''(t) - f'(t)}{t^2} \ge 0, 
    \end{equation}
    concluding that $h$ is increasing on $(0, \infty)$.
\end{proof}
\begin{lemma}
  \label{lem:mixture_simple}
  Let $f : \R \to \R$ be an even, twice continuously differentiable convex function whose second derivative $f''$ is convex. Let $\xi$ be a Rademacher random variable. Given a real number $r$, define the function
    \begin{equation}
      \label{eq:def_fr}
      f_r (t) = \E f(t + r\xi) .
    \end{equation}
    Then the function $f_r$ is even and convex, with $f_r''$ being convex.
\end{lemma}
\begin{proof}
  Since $f_r(t) = \frac{1}{2} f(t + r) + \frac{1}{2} f(t - r)$, the evenness of $f_r$ follows from $f(-t \pm r) = f(t \mp r)$. Furthermore, $f_r$ and its second derivative $f_r''(t) = \frac{1}{2} f''(t + r) + \frac{1}{2} f''(t - r)$ are convex as non-negative linear combinations of convex functions.
\end{proof}
\begin{lemma} 
  \label{lem:single}
    Let $r \geq 0$ and consider a zero-mean random vector $X$ taking values in the closed ball $B_r \subset \R^2$. Let $f : \R \to \R$ be an even, twice continuously differentiable convex function whose second derivative $f''$ is convex. Let $\vv \in \R^2$ be an arbitrary deterministic vector and $\xi$ a Rademacher random variable. 
    \begin{equation}
      \label{eq:single}
      \E f (\norm{\vv + X}) \le \E f \left(\norm{\vv} + r\xi\right) %= \E f \left(\norm{\vv\left(1 + \tfrac{r}{\norm{\vv}} \xi\right)}\right).
    \end{equation}
    %warning: the commented-out claims fail at some specific cases, consider $f(t) = t^2$
%    If $\vv$ is nonzero, equality holds whenever $X$ is concentrated on $
%    \{\pm \frac{r}{\norm{\vv}}\vv\}$.
%
%    If $\vv$ is zero then the equality is attained by any $X$ taking values on the circle of radius $r$ around zero.
\end{lemma}
\begin{proof}
    If $r=0$, then $X=0$ almost surely, and the assertion holds with equality. Hence we may assume that $r>0$.

    We claim that it is enough to prove the lemma in the case $r = 1$. Defining $f^*(t) = f(rt)$ we note that $f^*$ inherits the properties of $f$. Writing $\vv^* = \vv/r$ and $X^* = X/r$, we have that $X^*$ takes values in $B_1$ and therefore
\[
  \E f(\norm{\vv + X}) = \E f^*(\norm{\vv^* + X^*}) \le \E f^*({\norm{\vv^*} + \xi}) = \E f({\norm{\vv} + r\xi}).
\]
Since the Euclidean norm is invariant under rotations, we can w.l.o.g.\ assume that $\vv = (v, 0)$ with $v = \norm{\vv}$. Our goal is to show that the function $\xx \mapsto f(\norm{\vv + \xx})$ is, pointwise on $B_1$, dominated by the affine function 
\[
  l(\xx) = \alpha x_1 + \beta,
\]
where $\alpha$ and $\beta$ are chosen so that $l(\xx) = f(\norm{\vv + \xx})$ at points $(-1,0)$ and $(1,0)$, that is
\[
    \alpha = \tfrac{1}{2} \left( f(v+1) - f(v-1) \right), \qquad \beta = \tfrac{1}{2} \left( f(v+1) + f(v-1) \right) = \E f(v + \xi) \,.
\]
By assumption, function $f$ is convex and increasing on $[0, \infty)$. Since the norm is a convex function on $\R^2$, it easily follows that $\xx \mapsto f(\norm{\vv + \xx})$ is convex as well. Consequently, the difference $D(\xx) := f(\norm{\vv + \xx}) - l(\xx)$ is a convex function on $\R^2$. 

A fundamental property of convex functions is that on any compact convex set, such as the unit disk $B_1$, the maximum is attained on the boundary. Therefore, to prove that $D(\xx) \le 0$ for all $\xx \in B_1$, it is sufficient to show that the maximum of $D(\xx)$ on the unit circle is zero. By symmetry, it is enough to consider the upper semicircle.

If $v = 0$, then trivially $D(\xx) = f(1) - f(1) = 0$ on the whole semicircle. Let us assume that $v > 0$ for the rest of the proof that $D(\xx) \le 0$.
    
    We parametrize $\xx = (u, \sqrt{1 - u^2})$, $u \in [-1,1]$ so that
    \[
      l(\xx) = \alpha u + \beta, \quad \text{and} \quad \norm{\vv + \xx} = \sqrt{(u + v)^2 + 1 - u^2} = \sqrt{v^2 + 1 + 2vu} =: g(u).
    \]

    Note that $v > 0$ implies that $g$ strictly increases from $g(-1) = |v-1|$ to $g(1) = v + 1$ with derivative
    \[
      g'(u) = \frac{v}{g(u)} > 0, \qquad u \in (-1,1).
    \]

    We define a function $\Delta : [-1, 1] \to \R$ by setting
    \[
      \Delta(u) := f(\norm{\vv + \xx}) - l(\xx) = f(g(u)) - \alpha u - \beta.
    \]
    Our goal is to show that $\Delta \le 0$.
    Since $g(u)$ is a continuous function with values in a bounded interval (say $[0, v + 1]$), function $\Delta$ is continuous on $[-1,1]$.
    Observe that $\Delta(1) = f(v+1) - f(v+1) = 0$, and, using the evenness of $f$
    \[
      \Delta(-1) = f(|v-1|) + \alpha - \beta = f(|v-1|) - f(v-1) = 0.
    \]
    Moreover, $\Delta$ is continuously differentiable on $(-1,1)$. Namely, for $u \in (-1, 1)$,
    \begin{align*}
       \Delta'(u) = (f( g(u) ) - \alpha u - \beta)'= f'(g(u))g'(u) - \alpha = \frac{vf'(g(u))}{g(u)} - \alpha.
    \end{align*}
    Denoting  $h(t) = f'(t)/t$, we have shown
    \begin{equation}
        \Delta'(u) = vh(g(u)) - \alpha, \qquad u \in (-1,1).
    \end{equation}
    By Lemma~\ref{lem:increasing} function $h$ is increasing on $(0,\infty)$ and, recalling that $g$ is strictly positive and  increasing on $(-1,1)$, we conclude that $\Delta'$ is increasing on $(-1,1)$. Together with continuity of $\Delta$ this implies that $\Delta$ is convex on $[-1,1]$. Recalling that $\Delta(-1) = \Delta(1) = 0$, we conclude that $\Delta \le 0$, implying that $l \ge f(\norm{\vv + \cdot})$ on the disk $B_1$ as discussed above.

    Now, recalling that $X$ has zero mean, we conclude that the inequality in \eqref{eq:single} holds:
    \[
    \E f(\norm{\vv + X}) \le \E l(X) = \alpha \E X_1 + \beta = \alpha \cdot 0 + \beta = \E f({v + \xi}) = \E f({\norm{\vv} + \xi}),
    \]
    which proves inequality~\eqref{eq:single}.
%    Finally, the equality in \eqref{eq:single} holds, because using
%    \(
%      \bars{\norm{\vv} + r\xi} = \bignorm{ \vv + \tfrac{r}{\norm{\vv}}\vv\xi }
%    \)
%    and the evenness of $f$, we obtain
%    \[
%       f\left( \norm{\vv} + r\xi \right) = f\left( \bars{\norm{\vv} + r\xi} \right) = f\left(\bignorm{ \vv + \tfrac{r}{\norm{\vv}}\vv\xi }  \right).
%    \]
%
%    Next suppose that the equality in \eqref{eq:single} is attained. Then $X$ is supported on the set of maximizers of $\Delta$. If $\vv$ is nonzero, that is, $v > 0$, then $g$ is strictly increasing. Condition that $f'(t)/t$ is increasing implies that $f$ is strictly increasing as well, so the set of maximizers consists of the points $(v \pm 1, 0)$.
%
%    If $v = 0$, then $\Delta'(u) \equiv -\alpha$, a constant, which means that the set of maximizers is the whole circle meaning that any $X$ taking values on the circle attains equality.
\end{proof}
%In the more general setting where martingale difference $X_i$ takes values in the ball of radius $r_i$, we define
%\[
%  f(t) = f_{d,k}(t) := \textstyle \E \left|t + \sum_{i = 1}^k r_i\xi_i\right|^d = 2^{-k} \sum_{I \subseteq [k]} \bars{t + \sum_{i \in I} r_i - \sum_{i \notin I} r_i }^d 
%\]
%Like before we get that $f_{d,k}'' = (d)_2 f_{d-2,k}$, so the fourth derivative is nonnegative and the third derivative is odd (in particular $f^{'''}(0) = 0$. 

\begin{proof}[Proof of Theorem~\ref{thm:moments}]
  We proceed by induction on $n$. For the base case $n = 0$, since $M_0 = 0$ and an empty sum equals $0$, both sides evaluate to $f(0)$, so equality holds.

  For $n \ge 1$, since $X_n$ has conditional mean zero given $M_{n-1}$ and satisfies $\norm{X_n}\leq a_n$ almost surely, Lemma~\ref{lem:single} applies conditionally on $M_{n-1}$. Apply the lemma with $\vv = M_{n-1}$, $X = X_n$, and $r = a_n$. This yields
  \[
    \Ec{f(\norm{M_n})}{M_{n-1}} = \Ec{f(\norm{M_{n-1} + X_n})}{M_{n-1}} \le f_{a_n}(\norm{M_{n-1}}),
  \]
  where $f_{a_n}(t) = \E f(t + a_n\xi_n)$ as defined in \eqref{eq:def_fr}. Taking expectations on both sides gives
  \[
    \E f(\norm{M_n}) \le \E f_{a_n}(\norm{M_{n-1}}).
  \]
  By Lemma~\ref{lem:mixture_simple}, the function $f_{a_n}$ is an even convex function whose second derivative $f_{a_n}''$ is convex. Therefore, applying the induction hypothesis to $f_{a_n}$ on the martingale $(M_k)_{k=0}^{n-1}$, we obtain
  \[
    \E f_{a_n}(\norm{M_{n-1}}) \le \E f_{a_n}\left( \sum_{i = 1}^{n-1} a_i \xi_i \right) = \E f\left( \sum_{i = 1}^n a_i \xi_i \right),
  \]
  which completes the proof.
\end{proof}

\section{Proofs}

To reduce the proof of Theorems~\ref{thm:subgaussian_tail} and~\ref{thm:largedeviation_tail} from an arbitrary Euclidean space to $\R^2$, we rely on the following reduction theorem of Kallenberg and Sztencel \cite{kallenberg}.

\begin{lemma}[\citeref{kallenberg}{Theorem~3.1}]
\label{lem:KS}
Let $M = (M_k)_{k=0}^n$ be a martingale taking values in an Euclidean space $E$ with increments $X_k = M_k - M_{k-1}$. Then there exists a $\R^2$-valued martingale $N = (N_k)_{k=0}^n$ with increments $Y_k = N_k - N_{k-1}$ such that
\[
  (\|N_0\|, \|N_1\|, \dots, \|N_n\|) \overset{d}{=} (\|M_0\|, \|M_1\|, \dots, \|M_n\|)
\]
and $(\|Y_1\|, \dots, \|Y_n\|) \overset{d}{=} (\|X_1\|, \dots, \|X_n\|)$.
\end{lemma}

In particular, if the increments of $M$ satisfy $\|X_i\| \le a_i$ almost surely for $i=1,\ldots,n$, then the increments of $N$ satisfy the same bounds. Moreover, $\|M_n\|$ and $\|N_n\|$ have the same distribution. Consequently, any tail bound for $\|M_n\|$ depending only on the increment bounds $a_1,\ldots,a_n$ holds in every Euclidean space if it holds in $\R^2$.

\begin{lemma}[\cite{Pinelis1994Hotelling}]
  \label{lem:comparison}
  Let $f:\mathbb{R}\to\mathbb{R}$ be an even, twice continuously differentiable function whose second derivative $f''$ is convex. Let $S_n=\sum_{i=1}^n a_i\xi_i$, where $\xi_1,\ldots,\xi_n$ are independent Rademacher random variables, and suppose that $\sigma_n^2=\sum_{i=1}^n a_i^2>0$. Then
  $$ \E f(S_n) \le \E f(Z), $$
  where $Z$ is a centered Gaussian random variable with variance $\sigma_n^2$.
\end{lemma}
\begin{proof}
	Since $\sigma_n^2>0$ and the class of even functions with a finite convex second derivative is closed under rescaling of the argument, we may assume without loss of generality that $\sigma_n^2 = 1$. In this case the lemma is precisely \citeref{Pinelis1994Hotelling}{Corollary 2.7}.	
\end{proof}
%\begin{lemma}[Tomas' version]\label{comparison}
%Let $f$ be a convex function with a finite convex second derivative $f''$. Then for a linear combination $S_n=r_{1}\xi_{1}+\cdots+r_{n}\xi_{n}$ with $\sum r_{i}^{2}=1$ we have
%$$\E f(S_n)\leq \E f(Z),$$
%where $Z$ is a standard Gaussian random variable.\\
%Furthermore, for the particular choice $f_{x}(t)=(|t|-h)_{+}^{3}/(x-h)^3$ we have
%$$\inf_{h<x}\E f_{x}(S_n)\leq \frac{2e^3}{9}\Pr(|Z|\geq x).$$
%\end{lemma}
The following standard inequality is a special case ($\alpha = 3$) of \citeref{BKZ}{Lemma~1.1}
\begin{lemma}
	\label{lem:inversion_Gaussian}
For every real number $x$ and an arbitrary Gaussian variable $Z$ we have
    \[
    \inf_{h < x} \frac{\E (Z-h)_+^3}{(x-h)^3} \le \frac{2 \e^3}{9} \Pr\left(Z \ge x\right).
    \]
\end{lemma}

\begin{proof}[Proof of Theorem~\ref{thm:subgaussian_tail}]
	In view of Lemma~\ref{lem:KS}, we can assume that martingale $M$ is $\R^2$-valued. For any $h < x$, define the function $f_h(t) = (t-h)_+^3 + (-t-h)_+^3$. Specifically, for $t \ge x$, we have $f_h(t) \ge (x-h)^3$, so
\[
  \Pr(\|M_n\| \ge x) = \E \mathbb{I}_{\{\|M_n\| \ge x\}} \le \frac{\E f_h(\|M_n\|)}{(x-h)^3}.
\]

The function $f_h$ is even and has derivative $f_h'(t) = 3(t-h)_+^2 - 3(-t-h)_+^2$. Its second derivative, $f_h''(t) = 6(t-h)_+ + 6(-t-h)_+$, is the sum of two convex functions and is therefore convex on $\mathbb{R}$. This allows us to apply Theorem~\ref{thm:moments} (assuming the martingale differences are indexed as $X_i = M_i - M_{i-1}$), yielding
\[
  \E f_h(\|M_n\|) \le \E f_h\left( S_n \right),
\]
where $S_n = \sum_{i=1}^n a_i \xi_i$ with $\xi_i$ being independent Rademacher random variables.

 Lemma~\ref{lem:comparison} gives
\[
  \E f_h(S_n) \le \E f_h(Z_n),
\]
where $Z_n \sim N(0, \sigma_n^2)$. 

Because the distribution of $Z_n$ is symmetric around zero, $(-Z_n-h)_+$ has the same distribution as $(Z_n-h)_+$, giving $\E f_h(Z_n) = 2 \E (Z_n-h)_+^3$. We obtain
\begin{align*}
  \Pr(\|M_n\| \ge x) &\le \inf_{h < x} \frac{\E f_h(Z_n)}{(x-h)^3} \\
  &= 2\inf_{h < x} \frac{\E (Z_n - h)_+^3}{(x - h)^3}.
\end{align*}

Applying Lemma~\ref{lem:inversion_Gaussian}, this is bounded by
\[
	2 \left( \frac{2\e^3}{9} \Pr(Z_n \ge x) \right) = \frac{2\e^3}{9} \Pr(|Z_n| \ge x),
\]
completing the proof.
\end{proof}
%Tomo juodraštis:
%\begin{proof}[Tomas' proof of Theorem~\ref{thm:subgaussian_tail}]
%The function $f_{x}(t)=(|t|-h)_{+}^{3}/(x-h)^3$ is convex, even and $f''\geq 0$. We have $\mathbb{I}\{|t|\geq x\}\leq f_{x}(|t|).$ Replacing $|t|$ by $||M_n||$, taking the expectation in the latter inequality and using Theorem \ref{thm:moments} we get
%$$\mathbb{P}(||M_n||\geq x)\leq \E f_{x}(||M_n||)\leq \E f_{x}(r_1\xi_{1}+\cdots+r_n\xi_{n}).$$
%We now use the second conclusion of Lemma \ref{lem:comparison} and obtain
%$$\Pr(||M_n||\geq x)\leq \frac{2e^3}{9}\Pr(|Z|\geq x).$$
%\end{proof}
%
%\begin{proof}[Proof of Theorem~\ref{thm:largedeviation_tail}]
%Much like in the proof of Theorem \ref{thm:subgaussian_tail} we start with the inequality
%$$\Pr(||M_n||\geq x)$$
%\end{proof}
\begin{proof}[Proof of Theorem~\ref{thm:largedeviation_tail}]
By Lemma~\ref{lem:KS}, we can assume without loss of generality that $M_n$ is an $\R^2$-valued martingale with increments $X_i = M_i - M_{i-1}$ satisfying $\|X_i\| \le a_i$ a.s.

For any $\lambda > 0$, consider the function $f(t) = 2 \cosh(\lambda t) = \e^{\lambda t} + \e^{-\lambda t}$. Since $f$ is even and $f''(t) = 2\lambda^2 \cosh(\lambda t)$ is convex, Theorem~\ref{thm:moments} implies
\begin{equation}
	\label{eq:cosh}
  \E \cosh(\lambda \|M_n\|) \le \E \cosh\left( \lambda \sum_{i=1}^n a_i \xi_i \right) = \prod_{i=1}^n \cosh(\lambda a_i),
\end{equation}
where $\xi_1, \dots, \xi_n$ are independent Rademacher random variables. By Corollary 1 and Example 2 of \cite{Eaton1970}, this product, subject to the constraint $\sum_{i=1}^n a_i^2 = n a^2$, is maximized when all coefficients are equal to the averaged parameter $a = \sqrt{\frac{1}{n}\sum_{i=1}^n a_i^2}$:
\[
  \prod_{i=1}^n \cosh(\lambda a_i) \le \cosh^n(\lambda a).
\]

Taking expectations on the sides of the inequality $\mathbb{I}_{\{\|M_n\| \ge x\}} \le 2\e^{-\lambda x} \cosh(\lambda \|M_n\|)$,  we obtain the exponential moment bound
\[
  \Pr(\|M_n\| \ge x) \le 2 \exp\left( -\left(\lambda x - n \ln \cosh(\lambda a)\right) \right).
\]

Optimizing over $\lambda > 0$ yields the Cramér transform (a.k.a. Fenchel-Legendre dual) of the logarithmic moment generating function for symmetric Bernoulli/Rademacher random variables. Following the standard derivation in \citeref{Boucheron2013}{Section 2.2}, we substitute $u = \frac{x}{na} \in [0, 1]$ to find
\[
  \sup_{\lambda > 0} \left\{ \lambda x - n \ln \cosh(\lambda a) \right\} = {n} h(u),
\]
where $h(u) = [(1+u)\ln(1+u) + (1-u)\ln(1-u)]/2$. For $0<x<na$, the supremum is attained at $\lambda=a^{-1}\operatorname{arctanh}(x/(na))$. For $x=na$, the asserted value follows by letting $\lambda\to\infty$, and the case $x=0$ is immediate. This concludes the proof.
\end{proof}

%Finally, we recover Pinelis' inequality in \eqref{eq:Pinelis} as follows. Note that if $f(t)=cosh(t)$ then
\medskip
Finally, we recover Pinelis' inequality in \eqref{eq:Pinelis} from \eqref{eq:cosh} as follows.
 Using the standard inequality $\cosh(u) \le \e^{u^2/2}$ for all $u \in \R$, we bound the product by
\[
  \prod_{i=1}^n \cosh(\lambda a_i) \le \exp\left( \frac{\lambda^2}{2} \sum_{i=1}^n a_i^2 \right).
\]
To bound the tail probability, we observe that for any $t \ge x$, the indicator function satisfies $\mathbb{I}_{\{t \ge x\}} \le \e^{\lambda(t - x)} \le 2 \e^{-\lambda x} \cosh(\lambda t)$. Applying this to $\|M_n\|$ and taking expectations yields
\[
  \Pr(\|M_n\| \ge x) \le 2 \e^{-\lambda x} \E \cosh(\lambda \|M_n\|) \le 2 \exp\left( -\lambda x + \frac{\lambda^2}{2} \sum_{i=1}^n a_i^2 \right).
\]
Assuming that $\sum_{i=1}^n a_i^2>0$, for $x>0$ we optimize this bound over $\lambda>0$ by choosing $\lambda=x/\sum_{i=1}^n a_i^2$, obtaining
\[
  \Pr(\|M_n\| \ge x) \le 2 \e^{-\frac{x^2}{2\sum_{i=1}^n a_i^2}},
\]
which is Pinelis' inequality. The case $x=0$ is immediate.

\section{Acknowledgements}
The questions and proofs were human generated. AI was used for proofreading. Supported (partly) by the long-term strategic development financing of the Institute of Computer Science (RVO: 67985807).

\bibliographystyle{plain}
\bibliography{norm_tails}
\end{document}